\documentclass[amssymb, 11pt]{amsart}
\usepackage{latexsym}

\newlength{\standardunitlength}
\newtheorem{prop}{Proposition}[section]

\newtheorem{lemma}[prop]{Lemma}

\newtheorem{theorem}[prop]{Theorem}

\begin{document}

\title [Cohen-Lenstra heuristics and random matrix theory] {Cohen-Lenstra heuristics and random matrix theory over finite fields}

\author{Jason Fulman}
\address{Department of Mathematics\\
        University of Southern California\\
        Los Angeles, CA, 90089}
\email{fulman@usc.edu}



\thanks{{\it Date}: June 29, 2013}

\thanks{Fulman was partially supported by NSA grant H98230-13-1-0219.}

\begin{abstract} Let $g$ be a random element of a finite classical group $G$, and let $\lambda_{z-1}(g)$ denote the
partition corresponding to the polynomial $z-1$ in the rational canonical form of $g$. As the rank of $G$ tends to infinity, $\lambda_{z-1}(g)$ tends to a partition distributed according to a Cohen-Lenstra type measure on partitions. We give sharp upper and lower bounds on the total variation distance between the random partition $\lambda_{z-1}(g)$ and the Cohen-Lenstra type measure.
\end{abstract}

\maketitle

\section{Introduction}

The study of conjugacy classes of random elements of a group is an active subject. Indeed, for symmetric groups two elements are conjugate if and only if they have the same cycle structure, so this amounts to the study of the cycle structure of random permutations. And for the unitary groups $U(n,\mathbb{C})$, two elements are conjugate if and only if they have the same set of eigenvalues, so this amounts to the study of eigenvalues of random matrices.

Motivated by these considerations (and unaware of the Cohen-Lenstra heuristics of number theory), the author, in a series of papers \cite{F2}, \cite{F3}, \cite{FG}, \cite{FST}, investigated the conjugacy classes of random elements of a finite classical group. Two elements of $GL(n,q)$ are conjugate if and only if they have the same rational canonical form \cite{H}. Moreover, if $g$ is an element of a finite classical group, there is a partition $\lambda_{z-1}(g)$ corresponding to the polynomial $z-1$ in the rational canonical form of $g$. Letting the rank of $G$ tend to infinity, we proved that this random partition has a limiting distribution. For example if $G=GL(n,q)$, one obtains the distribution $P_{GL}$ on the set of all partitions of all non-negative integers, which chooses $\lambda$ with probability
\[ P_{GL}(\lambda) = \prod_{i \geq 1} (1-1/q^i) \cdot \frac{1}{|Aut(\lambda)|}, \] where $Aut(\lambda)$ denotes the automorphism group of a finite abelian group of type $\lambda$. We became fascinated with the combinatorics of such random partitions arising from random matrix theory over finite fields. In the papers \cite{F2}, \cite{F3}, \cite{FG}, \cite{FST}, we linked them to the Hall-Littlewood polynomials of symmetric function theory, and developed and applied probabilistic algorithms for growing such random partitions.

We were delighted to recently learn from Lengler \cite{L} that our work on random partitions is related to the Cohen-Lenstra heuristics \cite{CL} of number theory. Indeed, Cohen and Lenstra study random partitions chosen with probability \[ \prod_{i \geq 1} (1-1/q^i) \cdot \frac{1}{|Aut(\lambda)|},\] exactly the same formula as in our construction in the $GL$ case. It is beyond the scope of this paper to offer any sort of survey of Cohen-Lenstra heuristics, but we can assure the reader that research in the area is active and ongoing, with contributions from Bhargava, Malle, Ellenberg, Venkatesh, Poonen, Rains, and many others. We can recommend the papers \cite{D}, \cite{EV}, and the many references therein.

We are confident that all of the random partitions studied in the current paper will turn out to be related to Cohen-Lenstra heuristics. Indeed, our random partitions in the symplectic case were recently rediscovered in the Cohen-Lenstra context \cite{Ac}.
One of the goals of the current paper is to collect in one place all of the formulas for random partitions arising from random matrices over finite fields. These are currently ``scattered'' in the literature.

A second goal of the current paper is to quantify the convergence of the random partitions $\lambda_{z-1}(g)$ to their limit distributions. Recall that the total variation distance $||P-Q||_{TV}$ between two probability distributions on a set $X$ is defined as \[ ||P-Q||_{TV} = \frac{1}{2} \sum_{x \in X} |P(x)-Q(x)|.\] Let $\Lambda_{GL,z-1,n}$ denote the measure on partitions of size at most $n$ arising by taking the partition corresponding to the polynomial $z-1$ in the rational canonical form of a random element of $GL(n,q)$. One of the results of this paper is the sharp bound:
\[ \frac{.38}{q^{n+1}} \leq ||P_{GL} - \Lambda_{GL,z-1,n}||_{TV} \leq
\frac{14}{q^{n+1}}.\] This is more explicit and sharper than a similar recent result of Maples \cite{Map}, though we note that
Maples' main interest is different than ours: he proves universality of the distribution $P_{GL}$ for matrix ensembles where the entries are iid, but not necessarily uniform. We prove similar sharp bounds for the finite unitary, symplectic, and orthogonal groups, in both odd and even characteristic (where things can differ).

In terms of future work, it would be worthwhile to further explore the connections in \cite{F2} made between symmetric function theory (Hall-Littlewood polynomials) and random partitions arising from finite classical groups. In fact in work complementary to ours, Okounkov \cite {O1}, \cite{O2}, \cite{O3} makes many interesting connections between symmetric function theory (but not Hall-Littlewood polynomials) and random partitions (but not Cohen-Lenstra type measures). In fact one of his constructions, namely the definition of random partitions from Macdonald polynomials, was made independently in \cite{F2}. It would be very interesting to adapt Okounkov's methods to our setting.

The organization of this paper is as follows. Section \ref{GL} treats random partitions arising from the finite general linear groups. The unitary case is treated in Section \ref{U} and the symplectic case is treated in Section \ref{Sp}. Section \ref{O} treats random partitions arising from the finite orthogonal groups, and is split into two subsections, which consider odd and even characteristic respectively.

\section{General linear groups} \label{GL}

The Cohen-Lenstra measure \cite{CL} is a probability distribution on the set of all partitions of all non-negative integers.
We denote this measure by $P_{GL}$, since analogs for other finite classical groups will be given in other sections. A formula for the measure $P_{GL}$ is:
\[ P_{GL}(\lambda) = \prod_{i \geq 1} (1-1/q^i) \cdot \frac{1}{|Aut(\lambda)|} .\] Here $Aut(\lambda)$ denotes the automorphism group of a finite abelian group of type $\lambda$. Page 181 of \cite{Mac} gives the following explicit formula:
\[ |Aut(\lambda)| = q^{\sum_i (\lambda'_i)^2} \prod_i (1/q)_{m_i(\lambda)} .\] Here $m_i(\lambda)$ is the number of parts of $\lambda$ of size $i$, and $\lambda'$ is the partition dual to $\lambda$ in the sense that $\lambda'_i = m_i(\lambda) + m_{i+1}(\lambda) + \cdots$. Also $(1/q)_j$ denotes $(1-1/q)(1-1/q^2) \cdots (1-1/q^j)$.
\[ \]
{\it Remark:} The measure $P_{GL}$ is the special case ($u=1$) of a measure studied in \cite{F2}, which chooses $\lambda$ with probability \[ \prod_{i \geq 1} (1-u/q^i) \cdot \frac{u^{|\lambda|}}{|Aut(\lambda)|} .\] This probability can be rewritten as
\[ \prod_{i \geq 1} (1-u/q^i) \cdot \frac{P_{\lambda}(u/q,u/q^2,u/q^3,\cdots;1/q)}{q^{n(\lambda)}},\]
where $P_{\lambda}$ denotes a Hall-Littlewood polynomial and $n(\lambda)=\sum_i {\lambda_i' \choose 2}$. This suggests the study of more general measures based on Macdonald polynomials; see \cite{F2} for details, and for probabilistic algorithms for generating partitions according to this measure.

\[ \]

Next recall the rational canonical form of an element $g \in GL(n,q)$. This is discussed at length in Chapter 6 of the
textbook \cite{H}, and corresponds to the following combinatorial data. To each monic non-constant irreducible polynomial $\phi$ over $\mathbb{F}_q$, one associates a partition (perhaps the trivial partition) $\lambda_{\phi}$ of some non-negative integer $|\lambda_{\phi}|$. Let $deg(\phi)$ denote the degree of $\phi$. The only restrictions necessary for this data to arise
 from an element of $GL(n,q)$ are that $|\lambda_z|=0$ and $\sum_{\phi} |\lambda_{\phi}| deg(\phi) = n$.

We are interested in the partition (of size at most $n$) corresponding to the polynomial $z-1$ in the rational canonical form of a random element of $GL(n,q)$, and we let $\Lambda_{GL,z-1,n}$ denote the corresponding measure on partitions. From \cite{F0}, it is known that as $n \rightarrow \infty$, the measure $\Lambda_{GL,z-1,n}$ converges to the measure $P_{GL}$. The main result of this section gives sharp bounds for this convergence.

The next lemma is due to Euler; see page 19 of \cite{An}.

\begin{lemma} \label{eul}
\begin{enumerate}
\item $\prod_{i \geq 1} (1-u/q^i) = \sum_{j \geq 0} \frac{(-u)^j}{(q^j-1) \cdots (q-1)}$.
\item $\prod_{i \geq 1} (1-u/q^i)^{-1} = \sum_{j \geq 0} \frac{u^j q^{{j \choose 2}}}{(q^j-1) \cdots (q-1)}$.
\end{enumerate}
\end{lemma}

A proof of the following lemma is contained in \cite{St}.

\begin{lemma} \label{sto}
\[ \sum_{\lambda} \frac{u^{|\lambda|}}{|Aut(\lambda)|} = \prod_{i \geq 1} (1-u/q^i)^{-1} .\]
\end{lemma}

Next we obtain an explicit expression for $||P_{GL} - \Lambda_{GL,z-1,n}||_{TV}$.

\begin{prop} \label{explicit}
\begin{eqnarray*}
& & ||P_{GL} - \Lambda_{GL,z-1,n}||_{TV} \\
& = & \frac{1}{2} \sum_{m>n} \frac{q^{{m \choose 2}}}{(q^m-1) \cdots (q-1)} \prod_i (1-1/q^i) \\
& & + \frac{1}{2} \sum_{m=0}^n \frac{q^{{m \choose 2}}}{(q^m-1) \cdots (q-1)}
\left| \prod_i (1-1/q^i) - \sum_{j=0}^{n-m} \frac{(-1)^j}{(q^j-1) \cdots (q-1)} \right|.
\end{eqnarray*}
\end{prop}

\begin{proof} First we consider the contribution to the total variation distance coming from $\lambda$ of size $m>n$. Since $\Lambda_{GL,z-1,n}(\lambda)=0$ for such $\lambda$, the contribution is \[ \frac{1}{2} \sum_{|\lambda|=m>n} \frac{\prod_i (1-1/q^i)}{|Aut(\lambda)|}.\] By Lemma \ref{sto} and part 2 of Lemma \ref{eul}, this is equal to
\[ \frac{1}{2} \sum_{m>n} \frac{q^{{m \choose 2}}}{(q^m-1) \cdots (q-1)} \prod_i (1-1/q^i).\]

Next consider the probability that $\Lambda_{GL,z-1,n}$ associates to $\lambda$ when $|\lambda|=m \leq n$. By the cycle index of the general linear groups \cite{F1}, this probability is equal to the coefficient of $u^{n-m}$ in
\[ \frac{1}{|Aut(\lambda)|} \prod_{\phi \neq z,z-1} \left[ \sum_{\mu} \frac{u^{deg(\phi) |\mu|}}{|Aut(\mu)|_{q \rightarrow q^{deg(\phi)}}} \right]. \] Multiplying and dividing by $\sum_{\mu} \frac{u^{|\mu|}}{|Aut(\mu)|}$ gives that this is equal to the coefficient of $u^{n-m}$ in
\begin{eqnarray*}
& & \frac{1}{|Aut(\lambda)|} \frac{1}{\sum_{\mu} \frac{u^{|\mu|}}{|Aut(\mu)|}}
\prod_{\phi \neq z} \left[ \sum_{\mu} \frac{u^{deg(\phi) |\mu|}}{|Aut(\mu)|_{q \rightarrow q^{deg(\phi)}}}. \right] \\
& = & \frac{1}{|Aut(\lambda)|} \frac{1}{\sum_{\mu} \frac{u^{|\mu|}}{|Aut(\mu)|}} \frac{1}{1-u}.
\end{eqnarray*} The equality followed since setting all variables equal to one in the cycle index of $GL(n,q)$ yields
$1/(1-u)$. From Lemma \ref{sto}, this is the coefficient of $u^{n-m}$ in \[ \frac{1}{|Aut(\lambda)|} \frac{\prod_i (1-u/q^i)}{1-u}. \] Thus $\Lambda_{GL,z-1,n}(\lambda)$ is equal to \[ \frac{1}{|Aut(\lambda)|} \sum_{j=0}^{n-m} Coef. \ u^j \ in \ \prod_{i} (1-u/q^i),\] which by part 1 of Lemma \ref{eul} is equal to \[ \frac{1}{|Aut(\lambda)|} \sum_{j=0}^{n-m} \frac{(-1)^j}{(q^j-1) \cdots (q-1)} .\]

It follows that the contribution to $||P_{GL} - \Lambda_{GL,z-1,n}||_{TV}$ coming from $\lambda$ with
$|\lambda|=m \leq n$ is
\[ \frac{1}{2} \sum_{m=0}^n \sum_{|\lambda|=m} \frac{1}{|Aut(\lambda)|} \left| \prod_i (1-1/q^i) - \sum_{j=0}^{n-m}
\frac{(-1)^j}{(q^j-1) \cdots (q-1)} \right|.\] By Lemma \ref{sto} and part 2 of Lemma \ref{eul},
\[ \sum_{|\lambda|=m} \frac{1}{|Aut(\lambda)|} = Coef. \ u^m \ in \ \prod_{i} (1-u/q^i)^{-1} = \frac{q^{{m \choose 2}}}{(q^m-1) \cdots (q-1)} .\] Thus the contribution to $||P_{GL} - \Lambda_{GL,z-1,n}||_{TV}$ from $\lambda$ with
$|\lambda|=m \leq n$ is
\[ \frac{1}{2} \sum_{m=0}^n \frac{q^{{m \choose 2}}}{(q^m-1) \cdots (q-1)} \left|
\prod_i (1-1/q^i) - \sum_{j=0}^{n-m} \frac{(-1)^j}{(q^j-1) \cdots (q-1)} \right|.\] This completes the proof.
\end{proof}

Next we prove the main result of this section.

\begin{theorem} \label{glbound} For $n \geq 1$,
\[ \frac{.38}{q^{n+1}} \leq ||P_{GL} - \Lambda_{GL,z-1,n}||_{TV} \leq
\frac{14}{q^{n+1}}.\]
\end{theorem}

\begin{proof} To begin we consider the lower bound. By considering the $m=n+1$ term in Proposition \ref{explicit},
it follows that
\begin{eqnarray*}
||P_{GL} - \Lambda_{GL,z-1,n}||_{TV} & \geq & \frac{1}{2} \frac{q^{{n+1 \choose 2}}}{(q^{n+1}-1) \cdots (q-1)} \prod_i (1-1/q^i) \\
& = & \frac{1}{2q^{n+1}} (1-1/q^{n+2})(1-1/q^{n+3}) \cdots.
\end{eqnarray*} Since $n \geq 1$ and $q \geq 2$, this is at least $.38/q^{n+1}$. Note that we have used the bound
\begin{eqnarray*}
 \prod_{j \geq 3} (1-1/2^j) & = & \frac{8}{3} \prod_{j \geq 1} (1-1/2^j) \\
 & \geq & \frac{8}{3} (1-1/2-1/2^2+1/2^5+1/2^7-1/2^{12}-1/2^{15}) \\
 & \geq & .77,
\end{eqnarray*} which follows from Euler's pentagonal number theorem, stated on page 11 of \cite{An}.

Next we consider the upper bound. First note that
\begin{eqnarray*}
\frac{1}{2} \sum_{m>n} \frac{q^{{m \choose 2}}}{(q^m-1) \cdots (q-1)} \prod_i (1-1/q^i) & \leq & \frac{1}{2} \sum_{m>n} \frac{q^{{m \choose 2}}}{q^{{m+1 \choose 2}}} \\
& = & \frac{1}{2} \sum_{m>n} \frac{1}{q^m} \\
& = & \frac{1}{2 q^{n+1} (1-1/q)} \\
& \leq & \frac{1}{q^{n+1}}.
\end{eqnarray*} Second, note that by applying part 1 of Lemma \ref{eul} with $u=1$,
\begin{eqnarray*}
& & \frac{1}{2} \sum_{m=0}^n \frac{q^{{m \choose 2}}}{(q^m-1) \cdots (q-1)}
\left| \prod_i (1-1/q^i) - \sum_{j=0}^{n-m} \frac{(-1)^j}{(q^j-1) \cdots (q-1)} \right| \\
& \leq & \frac{1}{2} \sum_{m=0}^n \frac{q^{{m \choose 2}}}{(q^m-1) \cdots (q-1)} \frac{1}{(q^{n-m+1}-1) \cdots (q-1)} \\
& = & \frac{1}{2} \sum_{m=0}^n \frac{q^{{m \choose 2}}}{q^{{m+1 \choose 2}} (1-1/q) \cdots (1-1/q^m)} \\
& & \cdot \frac{1}{q^{{n-m+2 \choose 2}} (1-1/q) \cdots (1-1/q^{n-m+1})}.
\end{eqnarray*} Since $\prod_i (1-1/q^i)^{-1} \leq 3.5$, this is at most
\[ \frac{(3.5)^2}{2} \sum_{m=0}^n \frac{1}{q^{m+{n-m+2 \choose 2}}} \leq  \frac{(3.5)^2}{2} \frac{1}{q^{n+1}(1-1/q)} \leq
\frac{13}{q^{n+1}}.\] 

The upper bound of the theorem follows immediately by combining the bounds in the previous paragraph.
\end{proof}

\section{Unitary groups} \label{U}

Next we define a unitary analog $P_U$ of the Cohen-Lenstra measure on the set of all partitions of all non-negative integers.
A formula for the measure $P_U$ is:
\[ P_U(\lambda) = \prod_{i \geq 1} (1+1/(-q)^i) \frac{1}{|Aut_U(\lambda)|} .\]
Here $|Aut_U(\lambda)|$  is defined by the formula
\[ |Aut_U(\lambda)| = q^{\sum_i (\lambda'_i)^2} \prod_i (-1/q)_{m_i(\lambda)}, \] where $(-1/q)_j = (1+1/q)(1-1/q^2) \cdots (1-(-1)^j/q^j)$.

\[ \]
{\it Remark:} The measure $P_{U}$ is the special case ($u=1$) of a measure studied in \cite{F2}, which chooses $\lambda$ with probability \[ \prod_{i \geq 1} (1+u/(-q)^i) \cdot \frac{u^{|\lambda|}}{|Aut_U(\lambda)|} .\]

\[ \]

To define the probability measure $\Lambda_{U,z-1,n}$, we use, as in the $GL$ case, the theory of rational canonical forms. Given an element $g \in U(n,q)$, there is a partition $\lambda_{z-1}(g)$ of size at most $n$ associated to the polynomial $z-1$. When $g$ is chosen uniformly at random from $U(n,q)$, we let $\Lambda_{U,z-1,n}$ denote the corresponding measure on partitions. From \cite{F0}, it is known that as $n \rightarrow \infty$, the measure $\Lambda_{U,z-1,n}$ converges to the measure $P_U$. The main result of this section makes this quantitative, proving that \[ \frac{1}{6 q^{n+1}} \leq ||P_{U} - \Lambda_{U,z-1,n}||_{TV}  \leq \frac{3}{q^{n+1}}.\]

Lemma \ref{eulU} is obtained by replacing $u,q$ by $-u,-q$ in Lemma \ref{eul}.

\begin{lemma} \label{eulU}
\begin{enumerate}
\item $\prod_{i \geq 1} (1+u/(-q)^i) = \sum_{j \geq 0} \frac{(-1)^{{j+1 \choose 2}} u^j}{(q^j-(-1)^j) \cdots (q+1)}$.
\item $\prod_{i \geq 1} (1+u/(-q)^i)^{-1} = \sum_{j \geq 0} \frac{u^j q^{{j \choose 2}}}{(q^j-(-1)^j) \cdots (q+1)}$.
\end{enumerate}
\end{lemma}

The next lemma is a unitary analog of Lemma \ref{sto}.

\begin{lemma} \label{stoU}
\[ \sum_{\lambda} \frac{u^{|\lambda|}}{|Aut_U(\lambda)|} = \prod_{i \geq 1} (1+u/(-q)^i)^{-1}.\]
\end{lemma}

\begin{proof} From the formulas for $|Aut(\lambda)|$ and $|Aut_U(\lambda)|$, one checks that
\[ \sum_{\lambda} \frac{u^{|\lambda|}}{|Aut_U(\lambda)|} = \sum_{\lambda} \frac{(-u)^{|\lambda|}}{|Aut(\lambda)|_{q \rightarrow -q}} .\]
The result now follows from Lemma \ref{sto}, since setting $u \rightarrow -u, q \rightarrow -q$ in $\prod_i (1-u/q^i)^{-1}$ yields $\prod_i (1+u/(-q)^i)^{-1}$.
\end{proof}

Next we give an explicit expression for $||P_U - \Lambda_{U,z-1,n}||_{TV}$.

\begin{prop} \label{explicitU}
\begin{eqnarray*}
& & ||P_{U} - \Lambda_{U,z-1,n}||_{TV} \\
& = & \frac{1}{2} \sum_{m>n} \frac{q^{{m \choose 2}}}{(q^m-(-1)^m) \cdots (q+1)} \prod_i (1+1/(-q)^i) \\
& & + \frac{1}{2} \sum_{m=0}^n \frac{q^{{m \choose 2}}}{(q^m-(-1)^m) \cdots (q+1)} \\
& & \cdot \left| \prod_i (1+1/(-q)^i) - \sum_{j=0}^{n-m} \frac{(-1)^{{j+1 \choose 2}}}{(q^j-(-1)^j) \cdots (q+1)} \right|.
\end{eqnarray*}
\end{prop}

\begin{proof} First we consider the contribution to the total variation distance coming from $\lambda$ of size $m>n$. Since $\Lambda_{U,z-1,n}(\lambda)=0$ for such $\lambda$, the contribution is \[ \frac{1}{2} \sum_{|\lambda|=m>n} \frac{\prod_i (1+1/(-q)^i)}{|Aut_U(\lambda)|}. \] By Lemma \ref{stoU} and part 2 of Lemma \ref{eulU}, this is equal to
\[ \frac{1}{2} \sum_{m>n} \frac{q^{{m \choose 2}}}{(q^m-(-1)^m) \cdots (q+1)} \prod_i (1+1/(-q)^i).\]

Next consider the probability that $\Lambda_{U,z-1,n}$ associates to $\lambda$ when $|\lambda|=m \leq n$. By the cycle
index of the unitary groups \cite{F1}, this probability is equal to the coefficient of $u^{n-m}$ in
\[ \frac{1}{|Aut_U(\lambda)|} \frac{(1-u)^{-1}}{\sum_{\lambda} \frac{u^{|\lambda|}}{|Aut_U(\lambda)|}}. \] Indeed, \[ \frac{(1-u)^{-1}}{\sum_{\lambda} \frac{u^{|\lambda|}}{|Aut_U(\lambda)|}} \] is the part of the cycle index of the unitary groups corresponding to polynomials other than $z-1$. By Lemma \ref{stoU}, it follows that $\Lambda_{U,z-1,n}(\lambda)$ is equal to the coefficient of $u^{n-m}$ in \[ \frac{1}{|Aut_U(\lambda)|} \frac{\prod_i (1+u/(-q)^i)}{1-u},\] and hence equal to
\[ \frac{1}{|Aut_U(\lambda)|} \sum_{j=0}^{n-m} Coef. \ u^j \ in \ \prod_i (1+u/(-q)^i). \]
By part 1 of Lemma \ref{eulU}, this is equal to \[ \frac{1}{|Aut_U(\lambda)|} \sum_{j=0}^{n-m} \frac{(-1)^{{j+1 \choose 2}}}{(q^j-(-1)^j) \cdots (q+1)}.\]

Thus the contribution to $||P_{U} - \Lambda_{U,z-1,n}||_{TV}$ coming from $\lambda$ with $|\lambda|=m \leq n$ is
\[ \frac{1}{2} \sum_{m=0}^n \sum_{|\lambda|=m} \frac{1}{|Aut_U(\lambda)|} \left| \prod_i (1+1/(-q)^i) - \sum_{j=0}^{n-m}
\frac{(-1)^{{j+1 \choose 2}}}{(q^j-(-1)^j) \cdots (q+1)} \right|.\] By Lemma \ref{stoU} and part 2 of Lemma \ref{eulU},
\begin{eqnarray*}
\sum_{|\lambda|=m} \frac{1}{|Aut_U(\lambda)|} & = & Coef. \ u^m \ in \ \prod_i (1+u/(-q)^i)^{-1} \\
& = & \frac{q^{{m \choose 2}}}{(q^m-(-1)^m) \cdots (q+1)}.
\end{eqnarray*}
Thus the contribution to $||P_{U} - \Lambda_{U,z-1,n}||_{TV}$ from $\lambda$ with $|\lambda|=m \leq n$ is
\begin{eqnarray*}
& & \frac{1}{2} \sum_{m=0}^n \frac{q^{{m \choose 2}}}{(q^m-(-1)^m) \cdots (q+1)} \\
& & \cdot \left| \prod_i (1+1/(-q)^i) - \sum_{j=0}^{n-m} \frac{(-1)^{{j+1 \choose 2}}}{(q^j-(-1)^j) \cdots (q+1)} \right|,
\end{eqnarray*} and the proof is complete.
\end{proof}

Next we prove the main result of this section.

\begin{theorem} \label{Ubound} For $n \geq 1$,
\[ \frac{1}{6 q^{n+1}} \leq ||P_{U} - \Lambda_{U,z-1,n}||_{TV}  \leq \frac{3}{q^{n+1}}.\]
\end{theorem}

\begin{proof} To start we examine the lower bound. By considering the $m=n+1$ term in Proposition \ref{explicitU}, it follows that
\begin{eqnarray*}
& & ||P_{U} - \Lambda_{U,z-1,n}||_{TV} \\
& \geq & \frac{1}{2} \frac{q^{{n+1 \choose 2}}}{q^{{n+2 \choose 2}}} \frac{\prod_i (1+1/(-q)^i)}{(1+1/q)(1-1/q^2) \cdots (1-(-1)^{n+1}/q^{n+1})} \\
& \geq & \frac{1}{2} \frac{q^{{n+1 \choose 2}}}{q^{{n+2 \choose 2}}} \frac{(1-1/q)}{(1+1/q)} \\
& = & \frac{(1-1/q)}{2(1+1/q)} \frac{1}{q^{n+1}} \\
& \geq & \frac{1}{6 q^{n+1}}.
\end{eqnarray*}

Next we treat the upper bound. First note that
\begin{eqnarray*}
& & \frac{1}{2} \sum_{m>n} \frac{q^{{m \choose 2}}}{(q^m-(-1)^m) \cdots (q+1)} \prod_i (1+1/(-q)^i) \\
& = & \frac{1}{2} \sum_{m>n} \frac{q^{{m \choose 2}}}{q^{{m+1 \choose 2}}} \frac{\prod_i (1+1/(-q)^i)}{(1+1/q) \cdots (1-(-1)^m/q^m)} \\
& \leq & \frac{1}{2} \sum_{m>n} \frac{q^{{m \choose 2}}}{q^{{m+1 \choose 2}}} \\
& \leq & \frac{1}{q^{n+1}}.
\end{eqnarray*}

Next note by part 1 of Lemma \ref{eulU} that
\begin{eqnarray*}
& & \left| \prod_i (1+1/(-q)^i) - \sum_{j=0}^{n-m} \frac{(-1)^{{j+1 \choose 2}}}{(q^j-(-1)^j) \cdots (q+1)} \right| \\
& = & \left| \sum_{j \geq n-m+1} \frac{(-1)^{{j+1 \choose 2}}}{(q^j-(-1)^j) \cdots (q+1)} \right| \\
& \leq & \sum_{j \geq n-m+1} \frac{1}{q^{{j+1 \choose 2}}} \\
& \leq & \frac{2}{q^{{n-m+2 \choose 2}}}.
\end{eqnarray*} Thus
\begin{eqnarray*}
& & \frac{1}{2} \sum_{m=0}^n \frac{q^{{m \choose 2}}}{(q^m-(-1)^m) \cdots (q+1)} \\
& & \cdot \left| \prod_i (1+1/(-q)^i) - \sum_{j=0}^{n-m} \frac{(-1)^{{j+1 \choose 2}}}{(q^j-(-1)^j) \cdots (q+1)} \right| \\
& \leq & \frac{1}{2} \sum_{m=0}^n \frac{q^{{m \choose 2}}}{q^{{m+1 \choose 2}} (1+1/q) \cdots (1-(-1)^m/q^m)} \frac{2}{q^{{n-m+2 \choose 2}}} \\
& \leq & \frac{1}{2} \sum_{m=0}^n \frac{1}{q^m} \frac{2}{q^{{n-m+2 \choose 2}}} \\
& \leq & \frac{1}{q^{n+1} (1-1/q)} \\
& \leq & \frac{2}{q^{n+1}}.
\end{eqnarray*}

Combining the bounds in the two previous paragraphs completes the proof.
\end{proof}

\section{Symplectic groups} \label{Sp}

Next we define a symplectic analog $P_{Sp}$ of the Cohen-Lenstra measure on the set of all partitions of all
non-negative integers in which the odd parts occur with even multiplicity. A formula for the measure $P_{Sp}$ is:
\[ P_{Sp}(\lambda) = \prod_{i \geq 1} (1-1/q^{2i-1}) \frac{1}{|Aut_{Sp}(\lambda)|}, \]
where $|Aut_{Sp}(\lambda)|$ is defined by the formula
\[ |Aut_{Sp}(\lambda)| = q^{n(\lambda)+ \frac{|\lambda|}{2} + \frac{o(\lambda)}{2}} \prod_i (1-1/q^2)(1-1/q^4) \cdots
    (1-1/q^{2 \lfloor \frac{m_i(\lambda)}{2} \rfloor}).\] Here, $m_i(\lambda)$ denotes the multiplicity of $i$ in the partition $\lambda$, $o(\lambda)$ denotes the number of odd parts of $\lambda$, and $n(\lambda) = \sum_{i} {\lambda_i' \choose 2}.$
\[ \]
{\it Remark:} The measure $P_{Sp}$ is the special case $(u=1)$ of a measure studied in \cite{F3}, which chooses $\lambda$ with probability
\[ \prod_{i \geq 1} (1-u^2/q^{2i-1}) \cdot \frac{u^{|\lambda|}}{|Aut_{Sp}(\lambda)|} .\]

\[ \]

To define the probability measure $\Lambda_{Sp,z-1,n}$, we use, as in the $GL$ and $U$ cases, the theory of rational canonical forms. Given an element
$g \in Sp(2n,q)$, there is a partition $\lambda_{z-1}(g)$ of size at most $2n$ associated to the polynomial $z-1$. When $g$ is chosen uniformly at random from $Sp(2n,q)$, we let $\Lambda_{Sp,z-1,n}$ denote the corresponding measure on partitions. From \cite{F3} (in odd characteristic) and \cite{FG} (in even characteristic), it is known that as $n \rightarrow \infty$, the measure
$\Lambda_{Sp,z-1,n}$ converges to the measure $P_{Sp}$. In fact the main result of this section is that
\[ \frac{.2}{q^{n+1}} \leq  ||P_{Sp} - \Lambda_{Sp,z-1,n} ||_{TV}  \leq \frac{2.5}{q^{n+1}}  .\]

Lemma \ref{eulSp} is obtained by replacing $u$ by $u^2 q$ and $q$ by $q^2$ in Lemma \ref{eul}.

\begin{lemma} \label{eulSp}
\begin{enumerate}
\item $\prod_{i \geq 1} (1-u^2/q^{2i-1}) = \sum_{j \geq 0} \frac{(-1)^j u^{2j} q^j}{(q^{2j}-1) \cdots (q^2-1)}$.

\item $\prod_{i \geq 1} (1-u^2/q^{2i-1})^{-1} = \sum_{j \geq 0} \frac{u^{2j} q^{j^2}}{(q^{2j}-1) \cdots (q^2-1)}$.
\end{enumerate}
\end{lemma}

The following lemma is from \cite{F3}.

\begin{lemma} \label{stoSp}
\[ \sum_{\lambda} \frac{u^{|\lambda|}}{|Aut_{Sp}(\lambda)|} = \prod_{i \geq 1} (1-u^2/q^{2i-1})^{-1} .\]
\end{lemma}

Now we give an explicit expression for $||P_{Sp} - \Lambda_{Sp,z-1,n} ||_{TV}$.

\begin{prop} \label{explicitSp}
\begin{eqnarray*}
& & ||P_{Sp} - \Lambda_{Sp,z-1,n} ||_{TV} \\
& = & \frac{1}{2} \sum_{m>n} \frac{q^{m^2}}{(q^{2m}-1) \cdots (q^2-1)} \prod_i (1-1/q^{2i-1}) \\
& & + \frac{1}{2} \sum_{m=0}^n \frac{q^{m^2}}{(q^{2m}-1) \cdots (q^2-1)} \\
& & \cdot \left|\prod_i (1-1/q^{2i-1}) - \sum_{j=0}^{n-m} \frac{(-1)^j q^j}{(q^{2j}-1) \cdots (q^2-1)} \right|.
\end{eqnarray*}
\end{prop}

\begin{proof} To begin we consider the contribution to the total variation distance coming from $\lambda$ of size
$2m>2n$. Since $\Lambda_{Sp,z-1,n}(\lambda)=0$ for such $\lambda$, the contribution is
\[ \frac{1}{2} \sum_{|\lambda|=2m>2n} \frac{\prod_i (1-1/q^{2i-1})}{|Aut_{Sp}(\lambda)|}. \]
By Lemma \ref{stoSp} and part 2 of Lemma \ref{eulSp}, this is equal to
\[ \frac{1}{2} \sum_{m>n} \frac{q^{m^2}}{(q^{2m}-1) \cdots (q^2-1)} \prod_i (1-1/q^{2i-1}) .\]

Next consider the probability that $\Lambda_{Sp,z-1,n}$ assigns to $\lambda$ when $|\lambda|=2m \leq 2n$.
By the cycle index of the symplectic groups, this probability is equal to the coefficient of $u^{2n-2m}$ in
\[ \frac{1}{|Aut_{Sp}(\lambda)|} \frac{(1-u^2)^{-1}}{\sum_{\lambda} \frac{u^{|\lambda|}}{|Aut_{Sp}(\lambda)|}}. \]
Indeed, \[ \frac{(1-u^2)^{-1}}{\sum_{\lambda} \frac{u^{|\lambda|}}{|Aut_{Sp}(\lambda)|}} \] is the part
of the cycle index of the symplectic groups corresponding to polynomials other than $z-1$. By
Lemma \ref{stoSp}, it follows that $\Lambda_{Sp,z-1,n}(\lambda)$ is equal to the coefficient of $u^{n-m}$
in \[ \frac{1}{|Aut_{Sp}(\lambda)|} \frac{\prod_i (1-u/q^{2i-1})}{1-u},\] and thus equal to
\[ \frac{1}{|Aut_{Sp}(\lambda)|} \sum_{j=0}^{n-m} Coef. \ u^j \ in \ \prod_i (1-u/q^{2i-1}).\]
By part 1 of Lemma \ref{eulSp}, this is equal to
\[ \frac{1}{|Aut_{Sp}(\lambda)|} \sum_{j=0}^{n-m} \frac{(-1)^j q^j}{(q^{2j}-1) \cdots (q^2-1)} .\]

Thus the contribution to $||P_{Sp} - \Lambda_{Sp,z-1,n} ||_{TV}$ coming from $\lambda$ with $|\lambda|=2m \leq 2n$ is
\[ \frac{1}{2} \sum_{m=0}^n \sum_{|\lambda|=2m} \frac{1}{|Aut_{Sp}(\lambda)|} \left|\prod_i (1-1/q^{2i-1}) - \sum_{j=0}^{n-m}
\frac{(-1)^j q^j}{(q^{2j}-1) \cdots (q^2-1)} \right| .\]
By Lemma \ref{stoSp} and part 2 of Lemma \ref{eulSp},
\begin{eqnarray*}
\sum_{|\lambda|=2m} \frac{1}{|Aut_{Sp}(\lambda)|} & = & Coef. \ u^{2m} \ in \ \prod_i (1-u^2/q^{2i-1})^{-1} \\
& = & \frac{q^{m^2}}{(q^{2m}-1) \cdots (q^2-1)}.
\end{eqnarray*} Thus the contribution to $||P_{Sp} - \Lambda_{Sp,z-1,n} ||_{TV}$ from $\lambda$ with $|\lambda|=2m \leq 2n$ is
\begin{eqnarray*}
& & \frac{1}{2} \sum_{m=0}^n \frac{q^{m^2}}{(q^{2m}-1) \cdots (q^2-1)} \\
& & \cdot \left|\prod_i (1-1/q^{2i-1}) - \sum_{j=0}^{n-m} \frac{(-1)^j q^j}{(q^{2j}-1) \cdots (q^2-1)} \right|,
\end{eqnarray*} which completes the proof.
\end{proof}

Now we prove the main result of this section.

\begin{theorem} \label{spbound} For $n \geq 1$,
\[ \frac{.2}{q^{n+1}} \leq  ||P_{Sp} - \Lambda_{Sp,z-1,n} ||_{TV}  \leq \frac{2.5}{q^{n+1}}  .\]
\end{theorem}

\begin{proof} To begin we treat the lower bound. By looking at the $m=n+1$ term in Proposition \ref{explicitSp},
it follows that
\begin{eqnarray*}
||P_{Sp} - \Lambda_{Sp,z-1,n} ||_{TV} & \geq & \frac{1}{2} \frac{q^{(n+1)^2}}{q^{2 {n+2 \choose 2}}} \prod_i (1-1/q^{2i-1}) \\
& \geq & \frac{.2}{q^{n+1}}.
\end{eqnarray*}

For the upper bound, first note that
\begin{eqnarray*}
& & \frac{1}{2} \sum_{m>n} \frac{q^{m^2}} {(q^{2m}-1) \cdots (q^2-1)} \prod_i (1-1/q^{2i-1}) \\
& = & \frac{1}{2} \sum_{m>n} \frac{1}{q^m} \frac{\prod_i (1-1/q^{2i-1})}{(1-1/q^2) \cdots (1-1/q^{2m})} \\
& \leq & \frac{1}{2} \sum_{m>n} \frac{1}{q^m} \\
& \leq & \frac{1}{q^{n+1}}.
\end{eqnarray*}

Next, note that by part 1 of Lemma \ref{eulSp} with $u=1$,
\begin{eqnarray*}
& & \frac{1}{2} \sum_{m=0}^n \frac{q^{m^2}}{(q^{2m}-1) \cdots (q^2-1)} \\
& & \cdot \left|\prod_i (1-1/q^{2i-1}) - \sum_{j=0}^{n-m} \frac{(-1)^j q^j}{(q^{2j}-1) \cdots (q^2-1)} \right| \\
& \leq & \frac{1}{2} \sum_{m=0}^n \frac{q^{m^2}}{(q^{2m}-1) \cdots (q^2-1)} \frac{q^{n-m+1}}{(q^{2(n-m+1)}-1) \cdots (q^2-1)}.
\end{eqnarray*} Since \[ \frac{1}{(1-1/q^2)(1-1/q^4) \cdots} \leq 1.5,\] the upper bound becomes
\[ \frac{(1.5)^2}{2} \sum_{m=0}^n \frac{1}{q^m} \frac{1}{q^{(n-m+1)^2}} \leq \frac{(1.5)^2}{2} \frac{1}{q^{n+1}(1-1/q^2)} \leq \frac{1.5}{q^{n+1}}.\]

Combining the bounds of the previous two paragraphs completes the proof.
\end{proof}

\section{Orthogonal groups} \label{O}

In treating the orthogonal groups it is necessary to separately consider the cases of odd and even characteristic. Subsection \ref{odd} treats odd characteristic,
and Subsection \ref{even} treats even characteristic.

\subsection{Odd characteristic} \label{odd}

Suppose throughout this subsection that the characteristic is odd. We define an orthogonal analog $P_O$ of the Cohen-Lenstra measure on the set of all partitions of all non-negative integers in which the even parts occur with even multiplicity. A formula for the measure $P_O$ is:

\[ P_O(\lambda)= \frac{\prod_{i \geq 1} (1-1/q^{2i-1})}{2 |Aut_O(\lambda)|} ,\] where $|Aut_O(\lambda)|$ is defined by the formula

\[ |Aut_O(\lambda)| = q^{n(\lambda)+ \frac{|\lambda|}{2} - \frac{o(\lambda)}{2}} \prod_i (1-1/q^2) (1-1/q^4) \cdots
    (1-1/q^{2 \lfloor \frac{m_i(\lambda)}{2} \rfloor}).\] Here, as earlier, $m_i(\lambda)$ denotes the multiplicity of $i$ in the partition $\lambda$, $o(\lambda)$ denotes the number of odd parts of $\lambda$, and $n(\lambda) = \sum_{i} {\lambda_i' \choose 2}$.

\[ \]

{\it Remark}: The measure $P_O$ is the special case $(u=1)$ of a measure studied in \cite{F3}, which chooses $\lambda$ with probability
\[ \frac{\prod_{i \geq 1} (1-u^2/q^{2i-1})}{(1+u)} \frac{u^{|\lambda|}}{|Aut_O(\lambda)|}. \]

\[ \]

To define the probability measure $\Lambda_{O,z-1,n}$, we use, as in the cases of other finite classical groups, the theory of rational canonical forms.
We choose an element $g$, with probability $1/2$ uniformly at random from $O^+(n,q)$ and with probability $1/2$ uniformly at random from $O^-(n,q)$. Then there is
a partition $\lambda_{z-1}(g)$ of size at most $n$ associated to the polynomial $z-1$. We let $\Lambda_{O,z-1,n}$ denote the corresponding measure on partitions.
From \cite{F3} it is known that as $n \rightarrow \infty$, the measure $\Lambda_{O,z-1,n}$ converges to the measure $P_O$. The main result of this section is a sharp error term for this convergence.

The following lemma is from \cite{F3}.

\begin{lemma} \label{stoO} Suppose that $q$ is odd. Then
\[ \sum_{\lambda} \frac{u^{|\lambda|}}{|Aut_O(\lambda)|} = \frac{1+u}{\prod_i (1-u^2/q^{2i-1})} .\]
\end{lemma}

Next we given an explicit expression for $||P_O - \Lambda_{O,z-1,n}||_{TV}$.

\begin{prop} \label{explicitO} Suppose that $q$ is odd. Then
\begin{eqnarray*}
& & ||P_O - \Lambda_{O,z-1,n}||_{TV} \\
& = & \frac{1}{4} \sum_{m>n \atop m \ even} \frac{q^{m^2/4}}{(q^m-1) \cdots (q^2-1)} \prod_i (1-1/q^{2i-1}) \\
& & + \frac{1}{4} \sum_{m>n \atop m \ odd} \frac{q^{(m-1)^2/4}}{(q^{m-1}-1) \cdots (q^2-1)} \prod_i (1-1/q^{2i-1}) \\
& & + \frac{1}{4} \sum_{m=0 \atop m \ even}^n \frac{q^{m^2/4}}{(q^m-1) \cdots (q^2-1)} \\
& & \cdot \left| \prod_i (1-1/q^{2i-1}) - \sum_{j=0}^{\lfloor (n-m)/2 \rfloor} \frac{(-1)^j q^j} {(q^{2j}-1) \cdots (q^2-1)} \right| \\
& & + \frac{1}{4} \sum_{m=0 \atop m \ odd}^n \frac{q^{(m-1)^2/4}}{(q^{m-1}-1) \cdots (q^2-1)} \\
& & \cdot
 \left| \prod_i (1-1/q^{2i-1}) - \sum_{j=0}^{\lfloor (n-m)/2 \rfloor} \frac{(-1)^j q^j} {(q^{2j}-1) \cdots (q^2-1)} \right|.
\end{eqnarray*}
\end{prop}

\begin{proof} To begin we consider the contribution to the total variation distance coming from $\lambda$ of size $m>n$.
Since $\Lambda_{O,z-1,n}(\lambda)=0$ for such $\lambda$, the contribution is
\[ \frac{1}{4} \sum_{|\lambda|=m>n} \frac{1}{|Aut_O(\lambda)|} \prod_i (1-1/q^{2i-1}). \] By Lemma \ref{stoO} and part 2 of Lemma \ref{eulSp},
this is equal to
\begin{eqnarray*}
& & \frac{1}{4} \sum_{m>n \atop m \ even} \frac{q^{m^2/4}}{(q^m-1) \cdots (q^2-1)} \prod_i (1-1/q^{2i-1}) \\
& & + \frac{1}{4} \sum_{m>n \atop m \ odd} \frac{q^{(m-1)^2/4}}{(q^{m-1}-1) \cdots (q^2-1)} \prod_i (1-1/q^{2i-1}).
\end{eqnarray*}

Next we consider the probability that $\Lambda_{O,z-1,n}$ assigns to $\lambda$ when $|\lambda|=m \leq n$. By the cycle index
of the orthogonal groups \cite{F1}, this is equal to
\[ \frac{1}{2} \frac{1}{|Aut_O(\lambda)|} Coef. \ u^{n-m} \ in \ \frac{1+u}{\prod_i (1-u^2/q^{2i-1})} \frac{\prod_i (1-u^2/q^{2i-1})^2}{1-u^2} .\]
Indeed, the term $\frac{1+u}{\prod_i (1-u^2/q^{2i-1})}$ corresponds to the polynomial $z+1$ in the cycle index, and the term $\frac{\prod_i (1-u^2/q^{2i-1})^2}
{1-u^2}$ corresponds to polynomials other than $z \pm 1$ in the cycle index. Canceling terms, one obtains that $\Lambda_{O,z-1,n}(\lambda)$ is equal
to
\begin{eqnarray*}
& & \frac{1}{2} \frac{1}{|Aut_O(\lambda)|} Coef. \ u^{n-m} \ in \ \frac{\prod_i (1-u^2/q^{2i-1})}{1-u} \\
& = & \frac{1}{2} \frac{1}{|Aut_O(\lambda)|} \sum_{j=0}^{n-m} Coef. \ u^j \ in \ \prod_i (1-u^2/q^{2i-1}) \\
& = & \frac{1}{2} \frac{1}{|Aut_O(\lambda)|} \sum_{j=0}^{\lfloor (n-m)/2 \rfloor} Coef. \ u^{2j} \ in \ \prod_i (1-u^2/q^{2i-1}) \\
& = & \frac{1}{2} \frac{1}{|Aut_O(\lambda)|} \sum_{j=0}^{\lfloor (n-m)/2 \rfloor} \frac{(-1)^j q^j}{(q^{2j}-1) \cdots (q^2-1)},
\end{eqnarray*} where the last step used part 1 of Lemma \ref{eulSp}.

Thus the contribution to $||P_O - \Lambda_{O,z-1,n}||_{TV}$ coming from $\lambda$ with $|\lambda|=m \leq n$ is
\[ \frac{1}{2} \sum_{m=0}^n \sum_{|\lambda|=m} \frac{1}{2 |Aut_O(\lambda)|} \left| \prod_i (1-1/q^{2i-1}) - \sum_{j=0}^{\lfloor (n-m)/2 \rfloor} \frac{(-1)^j q^j}
{(q^{2j}-1) \cdots (q^2-1)} \right|.\] By Lemma \ref{stoO} and part 2 of Lemma \ref{eulSp},
\[ \sum_{|\lambda|=m} \frac{1}{|Aut_O(\lambda)|} = \frac{q^{m^2/4}}{(q^m-1) \cdots (q^2-1)},\] if $m$ is even, and
\[ \sum_{|\lambda|=m} \frac{1}{|Aut_O(\lambda)|} = \frac{q^{(m-1)^2/4}}{(q^{m-1}-1) \cdots (q^2-1)},\] if $m$ is odd. Thus
the contribution to $||P_O - \Lambda_{O,z-1,n}||_{TV}$ coming from $\lambda$ with $|\lambda|=m \leq n$ is
\begin{eqnarray*}
& & \frac{1}{4} \sum_{m=0 \atop m \ even}^n \frac{q^{m^2/4}}{(q^m-1) \cdots (q^2-1)} \\
 & & \cdot \left| \prod_i (1-1/q^{2i-1}) - \sum_{j=0}^{\lfloor (n-m)/2 \rfloor} \frac{(-1)^j q^j} {(q^{2j}-1) \cdots (q^2-1)} \right| \\
& & + \frac{1}{4} \sum_{m=0 \atop m \ odd}^n \frac{q^{(m-1)^2/4}}{(q^{m-1}-1) \cdots (q^2-1)} \\
& & \cdot
 \left| \prod_i (1-1/q^{2i-1}) - \sum_{j=0}^{\lfloor (n-m)/2 \rfloor} \frac{(-1)^j q^j} {(q^{2j}-1) \cdots (q^2-1)} \right|.
\end{eqnarray*} This completes the proof.
\end{proof}

Now we prove the main result of this section.

\begin{theorem} \label{obound}
\begin{enumerate}
\item For $n \geq 2$ even and $q$ odd,
\[ \frac{.1}{q^{n/2}} \leq ||P_O - \Lambda_{O,z-1,n}||_{TV} \leq \frac{1.3}{q^{n/2}}.\]

\item For $n \geq 1$ odd and $q$ odd,
\[ \frac{.1}{q^{(n+1)/2}} \leq ||P_O - \Lambda_{O,z-1,n}||_{TV} \leq \frac{2}{q^{(n+1)/2}}.\]
\end{enumerate}
\end{theorem}

\begin{proof} Suppose that $n$ is even. To lower bound the total variation distance, looking at the $m=n+1$ term in Proposition \ref{explicitO} gives that
\begin{eqnarray*}
||P_O - \Lambda_{O,z-1,n}||_{TV} & \geq & \frac{1}{4} \frac{q^{n^2/4}}{(q^n-1) \cdots (q^2-1)} \prod_i (1-1/q^{2i-1}) \\
& \geq & \frac{1}{4} \frac{q^{n^2/4}}{q^{n^2/4+n/2}} \prod_i (1-1/q^{2i-1}) \\
& \geq & \frac{.1}{q^{n/2}}.
\end{eqnarray*}

Next we consider the upper bound when $n$ is even; by Proposition \ref{explicitO} this is a sum of four terms. The first term is
\begin{eqnarray*}
& & \frac{1}{4} \sum_{m>n \atop m \ even} \frac{q^{m^2/4}}{(q^m-1) \cdots (q^2-1)} \prod_i (1-1/q^{2i-1}) \\
& = & \frac{1}{4} \sum_{m>n \atop m \ even} \frac{1}{q^{m/2}} \frac{\prod_i (1-1/q^{2i-1})}{(1-1/q^2) \cdots (1-1/q^m)} \\
& \leq & \frac{1}{4} \sum_{m>n \atop m \ even} \frac{1}{q^{m/2}} \\
& = & \frac{1}{4} \frac{1}{q^{n/2+1} (1-1/q)} \\
& \leq & \frac{3}{8} \frac{1}{q^{n/2+1}} \\
& \leq & \frac{1}{8 q^{n/2}}.
\end{eqnarray*} The second term in the upper bound is
\begin{eqnarray*}
& & \frac{1}{4} \sum_{m>n \atop m \ odd} \frac{q^{(m-1)^2/4}}{(q^{m-1}-1) \cdots (q^2-1)} \prod_i (1-1/q^{2i-1}) \\
& = & \frac{1}{4} \sum_{m>n \atop m \ odd} \frac{1}{q^{(m-1)/2}} \frac{\prod_i (1-1/q^{2i-1})}{(1-1/q^2) \cdots (1-1/q^{m-1})} \\
& \leq & \frac{1}{4} \sum_{m>n \atop m \ odd} \frac{1}{q^{(m-1)/2}} \\
& = & \frac{1}{4} \frac{1}{q^{n/2} (1-1/q)} \\
& \leq & \frac{3}{8 q^{n/2}}.
\end{eqnarray*} The third term in the upper bound is
\begin{eqnarray*}
& &  \frac{1}{4} \sum_{m=0 \atop m \ even}^n \frac{q^{m^2/4}}{(q^m-1) \cdots (q^2-1)} \\
& & \cdot \left| \prod_i (1-1/q^{2i-1}) - \sum_{j=0}^{(n-m)/2} \frac{(-1)^j q^j} {(q^{2j}-1) \cdots (q^2-1)} \right| \\
& \leq & \frac{1}{4} \sum_{m=0 \atop m \ even}^n \frac{q^{m^2/4}}{(q^m-1) \cdots (q^2-1)} \frac{q^{(n-m+2)/2}}{(q^{n-m+2}-1) \cdots (q^2-1)}.
\end{eqnarray*} Since $\prod_j (1-1/q^{2j})^{-1} \leq 1.2$, the third term is at most
\begin{eqnarray*}
\frac{(1.2)^2}{4} \sum_{m=0 \atop m \ even}^n \frac{1}{q^{m/2}} \frac{1}{q^{(\frac{n-m+2}{2})^2}}
& \leq & \frac{(1.2)^2}{4 q^{n/2+1} (1-1/q)} \\
& \leq & \frac{3(1.2)^2}{8 q^{n/2+1}} \\
& \leq & \frac{.2}{q^{n/2}}.
\end{eqnarray*} The fourth term in the upper bound is
\begin{eqnarray*}
& & \frac{1}{4} \sum_{m=0 \atop m \ odd}^n \frac{q^{(m-1)^2/4}}{(q^{m-1}-1) \cdots (q^2-1)} \\
& & \cdot \left| \prod_i (1-1/q^{2i-1}) - \sum_{j=0}^{(n-m-1)/2} \frac{(-1)^j q^j}{(q^{2j}-1) \cdots (q^2-1)} \right| \\
& \leq & \frac{1}{4} \sum_{m=0 \atop m \ odd}^n \frac{q^{(m-1)^2/4}}{(q^{m-1}-1) \cdots (q^2-1)} \frac{q^{(n-m+1)/2}}{(q^{n-m+1}-1) \cdots (q^2-1)}.
\end{eqnarray*} Again using that $\prod_j (1-1/q^{2j})^{-1} \leq 1.2$, the fourth term is at most
\[ \frac{(1.2)^2}{4} \sum_{m=0 \atop m \ odd}^n \frac{1}{q^{(m-1)/2}} \frac{1}{q^{(\frac{n-m+1}{2})^2}}
\leq \frac{(1.2)^2}{4 q^{n/2} (1-1/q)} \leq \frac{.6}{q^{n/2}}.\]

Combining the above bounds proves the theorem for $n \geq 2$ even.

Next suppose that $n$ is odd. To lower bound the total variation distance, looking at the $m=n+1$ term in Proposition \ref{explicitO} gives that
\begin{eqnarray*}
||P_O - \Lambda_{O,z-1,n}||_{TV} & \geq & \frac{1}{4} \frac{q^{(n+1)^2/4}}{(q^{n+1}-1) \cdots (q^2-1)} \prod_i (1-1/q^{2i-1}) \\
& \geq & \frac{1}{4 q^{(n+1)/2}} \prod_i (1-1/q^{2i-1}) \\
& \geq & \frac{.1}{q^{(n+1)/2}}.
\end{eqnarray*}

By Proposition \ref{explicitO}, the upper bound is a sum of four terms. For the first term, one argues as in the $n$ even case to obtain
\begin{eqnarray*}
\frac{1}{4} \sum_{m>n \atop m \ even} \frac{q^{m^2/4}}{(q^m-1) \cdots (q^2-1)} \prod_i (1-1/q^{2i-1}) & \leq & \frac{1}{4} \sum_{m>n \atop m \ even}
\frac{1}{q^{m/2}} \\
& = & \frac{1}{4} \frac{1}{q^{(n+1)/2} (1-1/q)} \\
& \leq & \frac{3}{8 q^{(n+1)/2}}.
\end{eqnarray*} For the second term, one also argues as in the $n$ even case to obtain
\begin{eqnarray*}
\frac{1}{4} \sum_{m>n \atop m \ odd} \frac{q^{(m-1)^2/4}}{(q^{m-1}-1) \cdots (q^2-1)} \prod_i (1-1/q^{2i-1})
& \leq & \frac{1}{4} \sum_{m>n \atop m \ odd} \frac{1}{q^{(m-1)/2}} \\
& = & \frac{1}{4} \frac{1}{q^{(n+1)/2} (1-1/q)} \\
& \leq & \frac{3}{8 q^{(n+1)/2}}. \end{eqnarray*} The third term in the upper bound is
\begin{eqnarray*}
& & \frac{1}{4} \sum_{m=0 \atop m \ even}^n \frac{q^{m^2/4}}{(q^m-1) \cdots (q^2-1)} \\
& & \cdot \left| \prod_i (1-1/q^{2i-1}) - \sum_{j=0}^{(n-m-1)/2} \frac{(-1)^j q^j}{(q^{2j}-1) \cdots (q^2-1)} \right| \\
& \leq & \frac{1}{4} \sum_{m=0 \atop m \ even}^n \frac{q^{m^2/4}}{(q^m-1) \cdots (q^2-1)} \frac{q^{(n-m+1)/2}}{(q^{n-m+1}-1) \cdots (q^2-1)}.
\end{eqnarray*} Using $\prod_j (1-1/q^{2j})^{-1} \leq 1.2$ shows that the third term is at most
\begin{eqnarray*}
\frac{(1.2)^2}{4} \sum_{m=0 \atop m \ even}^n \frac{1}{q^{m/2}} \frac{1}{q^{(\frac{n-m+1}{2})^2}}
& \leq & \frac{.4}{q^{(n+1)/2} (1-1/q)} \\
& \leq & \frac{.6}{q^{(n+1)/2}}.
\end{eqnarray*} The fourth term in the upper bound is
\begin{eqnarray*}
& & \frac{1}{4} \sum_{m=0 \atop m \ odd}^n \frac{q^{(m-1)^2/4}}{(q^{m-1}-1) \cdots (q^2-1)} \\
& & \cdot \left| \prod_i (1-1/q^{2i-1}) - \sum_{j=0}^{(n-m)/2} \frac{(-1)^j q^j}{(q^{2j}-1) \cdots (q^2-1)} \right| \\
& \leq & \frac{1}{4} \sum_{m=0 \atop m \ odd}^n \frac{q^{(m-1)^2/4}}{(q^{m-1}-1) \cdots (q^2-1)} \frac{q^{(n-m+2)/2}}{(q^{n-m+2}-1) \cdots (q^2-1)}.
\end{eqnarray*} Using $\prod_j (1-1/q^{2j})^{-1} \leq 1.2$ shows that the fourth term is at most
\begin{eqnarray*}
\frac{(1.2)^2}{4} \sum_{m = 0 \atop m \ odd}^n \frac{1}{q^{(m-1)/2}} \frac{1}{q^{(\frac{n-m+2}{2})^2}}
& \leq & \frac{(1.2)^2}{4} \frac{1}{q^{(n+1)/2} (1-1/q)} \\
& \leq & \frac{.6}{q^{(n+1)/2}}.
\end{eqnarray*}

Combining the above bounds proves the theorem for $n \geq 1$ odd.
\end{proof}

\subsection{Even characteristic} \label{even}

Throughout this subsection it is assumed that the characteristic is even. We define an orthogonal analog $P_O$ of
the Cohen-Lenstra measure on the set of all partitions of all non-negative integers in which the odd parts occur with even multiplicity. A formula for the measure $P_O$ is:
\[ P_O(\lambda) = \frac{\prod_{i \geq 1} (1-1/q^{2i-1})}{2 |Aut_O(\lambda)|}, \] where $|Aut_O(\lambda)|$
is defined by the formula
\[ |Aut_O(\lambda)| = q^{n(\lambda)+\frac{|\lambda|}{2}+\frac{o(\lambda)}{2} - l(\lambda)} \prod_i (1-1/q^2) (1-1/q^4) \cdots (1-1/q^{2 \lfloor \frac{m_i(\lambda)}{2} \rfloor}).\] Here, as earlier, $m_i(\lambda)$ denotes the multiplicity of $i$ in the partition $\lambda$, $o(\lambda)$ denotes the number of odd parts of $\lambda$, and $n(\lambda) = \sum_{i} {\lambda_i' \choose 2}$. The symbol $l(\lambda)$ denotes the number of parts of $\lambda$.
\[ \]
{\it Remark}: The measure $P_O$ is the special case ($u=1$) of a measure studied in \cite{FST} which chooses $\lambda$ with probability \[ \frac{\prod_i (1-u^2/q^{2i-1})}{1+u^2} \frac{u^{|\lambda|}}{|Aut_O(\lambda)|}. \]

To define the probability measure $\Lambda_{O,z-1,n}$, we use, as in the other cases, the theory of rational canonical forms. We choose an element $g$, with probability $1/2$ uniformly at random from $O^+(2n,q)$ and with probability $1/2$ uniformly at random from $O^-(2n,q)$. (Note that in even characteristic, odd dimensional orthogonal groups are isomorphic to symplectic groups, so we focus on even dimensional orthogonal groups). Then there is a partition $\lambda_{z-1}(g)$ of size at most $2n$ associated to the polynomial $z-1$. We let $\Lambda_{O,z-1,n}$ denote the corresponding measure on partitions. From \cite{FST} it is known that as $n \rightarrow \infty$, the measure $\Lambda_{O,z-1,n}$ converges to the measure $P_O$. The main result of this section is to make this convergence quantitative.

The following lemma is from \cite{FST}.

\begin{lemma} \label{stoOeven} Suppose that $q$ is even. Then
\[ \sum_{\lambda} \frac{u^{|\lambda|}}{|Aut_O(\lambda)|} = \frac{1+u^2}{\prod_i (1-u^2/q^{2i-1})} .\]
\end{lemma}

Next we give an explicit expression for $||P_O - \Lambda_{O,z-1,n}||_{TV}$.

\begin{prop} \label{explicitOeven} Suppose that $q$ is even. Then
\begin{eqnarray*}
& & ||P_O - \Lambda_{O,z-1,n}||_{TV} \\
& = & \frac{1}{4} \sum_{m>n} \left[ \frac{q^{m^2}}{(q^{2m}-1) \cdots (q^2-1)} + \frac{q^{(m-1)^2}}{(q^{2m-2}-1) \cdots (q^2-1)}   \right] \\
& & \cdot \prod_i (1-1/q^{2i-1}) \\
& & + \frac{1}{4} \sum_{m=1}^n \left[ \frac{q^{m^2}}{(q^{2m}-1) \cdots (q^2-1)} +  \frac{q^{(m-1)^2}}{(q^{2m-2}-1) \cdots (q^2-1)}  \right] \\
& & \cdot \left| \prod_i (1-1/q^{2i-1}) -
\sum_{j=0}^{n-m} \frac{(-1)^j q^j}{(q^{2j}-1) \cdots (q^2-1)} \right| \\
& & + \frac{1}{4} \left| \prod_i (1-1/q^{2i-1}) -
\sum_{j=0}^{n} \frac{(-1)^j q^j}{(q^{2j}-1) \cdots (q^2-1)} \right|.
\end{eqnarray*}
\end{prop}

\begin{proof} To begin, we consider the contribution to the total variation distance coming from $\lambda$ of size
$2m>2n$. Since $\Lambda_{O,z-1,n}(\lambda)=0$ for such $\lambda$, the contribution is
\[ \frac{1}{4} \sum_{|\lambda|=2m>2n} \frac{\prod_i (1-1/q^{2i-1})}{|Aut_O(\lambda)|}. \] By Lemma \ref{stoOeven} and part 2 of Lemma \ref{eulSp}, this is equal to
\[ \frac{1}{4} \sum_{m>n} \left[ \frac{q^{m^2}}{(q^{2m}-1) \cdots (q^2-1)} + \frac{q^{(m-1)^2}}{(q^{2m-2}-1) \cdots (q^2-1)}   \right] \prod_i (1-1/q^{2i-1}).\]

Next consider the probability that $\Lambda_{O,z-1,n}$ associates to $\lambda$ when $|\lambda|=2m \leq 2n$. By the cycle index of the orthogonal groups \cite{FST}, this probability is equal to the coefficient of $u^{2n-2m}$ in
\[ \frac{1}{2 |Aut_O(\lambda)|} \frac{\prod_i (1-u^2/q^{2i-1})}{1-u^2}.\] Indeed, \[ \frac{\prod_i (1-u^2/q^{2i-1})}{1-u^2} \] is the part of the cycle index of the orthogonal groups corresponding to polynomials other than $z-1$. Thus $\Lambda_{O,z-1,n}(\lambda)$ is equal to
\[ \frac{1}{2 |Aut_O(\lambda)|} \sum_{j=0}^{n-m} Coef. \ u^j \ in \ \prod_i (1-u/q^{2i-1}).\] By part 1 of Lemma
\ref{eulSp}, this is equal to \[ \frac{1}{2 |Aut_O(\lambda)|} \sum_{j=0}^{n-m} \frac{(-1)^j q^j}{(q^{2j}-1) \cdots (q^2-1)}.\]

Thus the contribution to $||P_O - \Lambda_{O,z-1,n}||_{TV}$ coming from $\lambda$ with $|\lambda|=2m \leq 2n$ is
\[ \frac{1}{4} \sum_{m=0}^n \sum_{|\lambda|=2m} \frac{1}{|Aut_O(\lambda)|} \left| \prod_i (1-1/q^{2i-1}) -
\sum_{j=0}^{n-m} \frac{(-1)^j q^j}{(q^{2j}-1) \cdots (q^2-1)} \right|.\] By Lemma \ref{stoOeven} and part 2 of Lemma
\ref{eulSp}, if $m \geq 1$, then
\begin{eqnarray*}
\sum_{|\lambda|=2m} \frac{1}{|Aut_O(\lambda)|} & = & Coef. \ u^{2m} \ in \ \frac{1+u^2}{\prod_i (1-u^2/q^{2i-1})} \\
& = & \frac{q^{m^2}}{(q^{2m}-1) \cdots (q^2-1)} \\
& & + \frac{q^{(m-1)^2}}{(q^{2m-2}-1) \cdots (q^2-1)}.
\end{eqnarray*} Thus the contribution to $||P_O - \Lambda_{O,z-1,n}||_{TV}$ coming from $\lambda$ with $|\lambda|=2m$
 with $1 \leq m \leq n$ is
\begin{eqnarray*}
& & \frac{1}{4} \sum_{m=1}^n \left[ \frac{q^{m^2}}{(q^{2m}-1) \cdots (q^2-1)} +  \frac{q^{(m-1)^2}}{(q^{2m-2}-1) \cdots (q^2-1)}  \right] \\
& & \cdot \left| \prod_i (1-1/q^{2i-1}) -
\sum_{j=0}^{n-m} \frac{(-1)^j q^j}{(q^{2j}-1) \cdots (q^2-1)} \right|.
\end{eqnarray*} The contribution to $||P_O - \Lambda_{O,z-1,n}||_{TV}$ coming from $|\lambda|=0$ is
\[ \frac{1}{4} \left| \prod_i (1-1/q^{2i-1}) - \sum_{j=0}^{n} \frac{(-1)^j q^j}{(q^{2j}-1) \cdots (q^2-1)} \right|,\] which completes the proof.
\end{proof}

Next we prove the main result of this section.

\begin{theorem} \label{oboundEven} For $n \geq 1$ and $q$ even,
\[ \frac{.1}{q^n} \leq ||P_O - \Lambda_{O,z-1,n}||_{TV} \leq \frac{2.6}{q^n} \]
\end{theorem}

\begin{proof} To lower bound the total variation distance, looking at the $m=n+1$ term in Proposition
\ref{explicitOeven} gives that
\begin{eqnarray*}
||P_O - \Lambda_{O,z-1,n}||_{TV} & \geq & \frac{1}{4} \frac{q^{n^2}}{(q^{2n}-1) \cdots (q^2-1)} \prod_i (1-1/q^{2i-1}) \\
& \geq & \frac{1}{4} \frac{q^{n^2}}{q^{n(n+1)}} \prod_i (1-1/q^{2i-1}) \\
& \geq & \frac{.1}{q^n}.
\end{eqnarray*}

Next we consider the upper bound; by Proposition \ref{explicitOeven}, this is a sum of three terms. Since \[ \frac{q^{m^2}}{(q^{2m}-1) \cdots (q^2-1)} \leq \frac{q^{(m-1)^2}}{(q^{2m-2}-1) \cdots (q^2-1)}, \] the first term is at most \begin{eqnarray*}
& & \frac{1}{2} \sum_{m>n} \frac{q^{(m-1)^2}}{(q^{2m-2}-1) \cdots (q^2-1)} \prod_i (1-1/q^{2i-1}) \\
& = & \frac{1}{2} \sum_{m>n} \frac{1}{q^{m-1}} \frac{\prod_i (1-1/q^{2i-1})}{(1-1/q^2) \cdots (1-1/q^{2m-2})} \\
& \leq & \frac{1}{2} \sum_{m>n} \frac{1}{q^{m-1}} \\
& = & \frac{1}{2} \frac{1}{q^n (1-1/q)} \\
& \leq & \frac{1}{q^n}.
\end{eqnarray*}

To upper bound the second term, note that
\begin{eqnarray*}
& &  \frac{1}{4} \sum_{m=1}^n \left[ \frac{q^{m^2}}{(q^{2m}-1) \cdots (q^2-1)} +  \frac{q^{(m-1)^2}}{(q^{2m-2}-1) \cdots (q^2-1)}  \right] \\
& & \cdot \left| \prod_i (1-1/q^{2i-1}) -
\sum_{j=0}^{n-m} \frac{(-1)^j q^j}{(q^{2j}-1) \cdots (q^2-1)} \right| \\
& \leq & \frac{1}{2} \sum_{m=1}^m \frac{q^{(m-1)^2}}{(q^{2m-2}-1) \cdots (q^2-1)} \\
& & \cdot \left| \prod_i (1-1/q^{2i-1}) -
\sum_{j=0}^{n-m} \frac{(-1)^j q^j}{(q^{2j}-1) \cdots (q^2-1)} \right| \\
& \leq & \frac{1}{2} \sum_{m=1}^n \frac{q^{(m-1)^2}}{(q^{2m-2}-1) \cdots (q^2-1)}
\frac{q^{n-m+1}}{(q^{2(n-m+1)}-1) \cdots (q^2-1)}.
\end{eqnarray*} Since $\prod_j (1-1/q^{2j})^{-1} \leq 1.5$, the second term is at most
\[ \frac{(1.5)^2}{2} \sum_{m=1}^n \frac{1}{q^{m-1}} \frac{1}{q^{(n-m+1)^2}} \leq \frac{(1.5)^2}{2} \frac{1}{q^n(1-1/q^2)} \leq \frac{1.5}{q^n}.\]

To upper bound the third term, note that
\[ \frac{1}{4} \left| \prod_i (1-1/q^{2i-1}) -
\sum_{j=0}^{n} \frac{(-1)^j q^j}{(q^{2j}-1) \cdots (q^2-1)} \right| \leq \frac{1}{4} \frac{q^{n+1}}{(q^{2(n+1)}-1) \cdots (q^2-1)}\] Since $\prod_j (1-1/q^{2j})^{-1} \leq 1.5$, the third term is at most \[ \frac{1.5}{4} \frac{1}{q^{(n+1)^2}} \leq \frac{1.5}{32} \frac{1}{q^n} \leq \frac{.05}{q^n}.\]

Adding the upper bounds on the three terms completes the proof.
\end{proof}


\begin{thebibliography}{AAA}

\bibitem [Ac]{Ac} Achter, J., The distribution of class groups of function fields, {\it J. Pure Appl. Algebra} {\bf 204}
(2006), 316-333.

\bibitem [An]{An} Andrews, G., The theory of partitions, Addison-Wesley, Reading, Mass., 1976.

\bibitem [CL]{CL} Cohen, H. and Lenstra, H.W., Jr., Heuristics on class groups of number fields, in {\it Number theory, Noordwijerhout 1983}, 33-62, Lecture Notes in Math. 1068, Springer, Berlin, 1984.

\bibitem [D]{D} Delaunay, C., Heuristics on class groups and on Tate-Shafarevich groups: the magic of the Cohen-Lenstra
heuristics, in {\it Ranks of elliptic curves and random matrix theory}, 323-340, London Math. Soc. Lecture Notes Ser. 341, Cambridge Univ. Press, Cambridge, 2007.

\bibitem [EV]{EV} Ellenberg, J., and Venkatesh, A., Statistics of number fields and function fields, in {\it Proceedings of the International Congress of Mathematicians. Volume II}, 383-402, Hindustan Book Agency, New Delhi, 2010.

\bibitem [F0]{F0} Fulman, J., {\it Probability in the classical groups over finite fields: symmetric functions, stochastic algorithms and cycle indices}, Ph.D. thesis, Harvard University, 1997.

\bibitem [F1]{F1} Fulman, J., Cycle indices for the finite classical
groups, {\it J. Group Theory} {\bf 2} (1999), 251-289.

\bibitem [F2]{F2} Fulman, J., A probabilistic approach toward conjugacy classes in the finite general linear and
 unitary groups, {\it J. Algebra} {\bf 212} (1999), 557-590.

\bibitem [F3]{F3} Fulman, J., A probabilistic approach to conjugacy classes in the finite symplectic and orthogonal groups,
 {\it J. Algebra} {\bf 234} (2000), 207-224.

\bibitem [FG]{FG} Fulman, J. and Guralnick, R., Conjugacy class properties of the extension of $GL(n,q)$ generated by the inverse transpose involution, {\it J. Algebra} {\bf 275} (2004), 356-396.

\bibitem [FST]{FST} Fulman, J., Saxl, J., and Tiep, P. H., Cycle indices for finite orthogonal groups of even characteristic,
{\it Trans. Amer. Math. Soc.} {\bf 364} (2012), 2539-2566.

\bibitem [H]{H} Herstein, I.N., Topics in algebra, Second edition. Xerox College Publishing, Lexington, Mass.-Toronto, Ont., 1975.

\bibitem [L]{L} Lengler, J., The Cohen-Lenstra heuristic: methodology and results, {\it J. Algebra} {\bf 323} (2010), 2960-2976.

\bibitem [Mac]{Mac} Macdonald, I., Symmetric functions and Hall polynomials, Second edition. The Clarendon Press, New York, 1995.

\bibitem [Map]{Map} Maples, K., Cokernels of random matrices satisfy the Cohen-Lenstra heuristics, arXiv:1301.1239 (2013).

\bibitem [O1]{O1} Okounkov, A., The uses of random partitions, in {\it XIVth International congress on mathematical physics}, 379-403, World Sci. Publ., Hackensack, NJ, 2005.

\bibitem [O2]{O2} Okounkov, A., Symmetric functions and random partitions, in {\it Symmetric functions 2001: surveys of developments and perspectives}, 223-252, Kluwer Acad. Publ., Dordrecht, 2002.

\bibitem [O3]{O3} Okounkov, A., Infinite wedge and random partitions, {\it Selecta Math. (N.S.)} {\bf 7} (2001), 57-81.

\bibitem [St]{St} Stong, R., Some asymptotic results on finite vector spaces, {\it Adv. in Appl. Math.} {\bf 9} (1988), 167-199.

\end{thebibliography}
\end{document}